\documentclass{amsart}

\usepackage{amsmath,amsthm,amsfonts,amssymb,latexsym,verbatim,bbm}
\input xy
\xyoption{all}
\usepackage{hyperref}

\theoremstyle{plain}
\newtheorem{theorem}{Theorem}[section]
\newtheorem{thm}[theorem]{Theorem}
\newtheorem{lemma}[theorem]{Lemma}

\newtheorem{prop}[theorem]{Proposition}

\newtheorem{defn}[theorem]{Definition}

\newtheorem{cor}[theorem]{Corollary}

\newcommand{\iso}{\cong}
\newcommand{\arr}{\rightarrow}
\newcommand{\Id}{\mathbbm{1}}
\newcommand{\incl}{\hookrightarrow}

\DeclareMathOperator{\End}{End}

\newcommand{\R}{\mathbb{R}}
\newcommand{\C}{\mathbb{C}}
\newcommand{\Z}{\mathbb{Z}}

\newcommand{\eps}{\epsilon}

\newcommand{\mcB}{\mathcal{B}}

\newcommand{\mcF}{\mathcal{F}}
\newcommand{\mcG}{\mathcal{G}}

\newcommand{\mcL}{\mathcal{L}}

\newcommand{\mcU}{\mathcal{U}}

\newcommand{\mfg}{\mathfrak{g}}
\newcommand{\mfh}{\mathfrak{h}}

\newcommand{\mfn}{\mathfrak{n}}

\DeclareMathOperator{\tr}{tr}
\DeclareMathOperator{\vspan}{span}

\newcommand{\pd}[2]{\frac{\partial#1}{\partial#2}}

\DeclareMathOperator{\mfsl}{\mathfrak{sl}}

\DeclareMathOperator{\mfu}{\mathfrak{u}}
\DeclareMathOperator{\mfb}{\mathfrak{b}}

\DeclareMathOperator{\mfa}{\mathfrak{a}}

\DeclareMathOperator{\mfs}{\mathfrak{s}}

\DeclareMathOperator{\Gr}{Gr}
\DeclareMathOperator{\Ad}{Ad}
\DeclareMathOperator{\ad}{ad}

\DeclareMathOperator{\Switch}{Switch}

\DeclareMathOperator{\gr}{gr}
\DeclareMathOperator{\IC}{IC}
\DeclareMathOperator{\mult}{mult}
\DeclareMathOperator{\ch}{ch}
\newcommand{\deltabar}{\bar{\partial}}
\newcommand{\boxbar}{\overline{\square}}
\DeclareMathOperator{\curv}{Curv}
\DeclareMathOperator{\degree}{deg}

\begin{document}

\title{A Brylinski filtration for affine Kac-Moody algebras}
\author{William Slofstra}
\begin{abstract}
    Braverman and Finkelberg have recently proposed a conjectural analogue of
    the geometric Satake isomorphism for untwisted affine Kac-Moody groups. As
    part of their model, they conjecture that (at dominant weights) Lusztig's
    $q$-analog of weight multiplicity is equal to the Poincare series of the
    principal nilpotent filtration of the weight space, as occurs in the
    finite-dimensional case. We show that the conjectured equality holds for
    all affine Kac-Moody algebras if the principal nilpotent filtration is
    replaced by the principal Heisenberg filtration. The main body of the proof
    is a Lie algebra cohomology vanishing result.  We also give an example to
    show that the Poincare series of the principal nilpotent filtration is not
    always equal to the $q$-analog of weight multiplicity.  Finally, we give
    some partial results for indefinite Kac-Moody algebras. 
\end{abstract}
\maketitle

\section{Introduction}

Let $\mcL(\lambda)$ be an integrable highest-weight representation of a
symmetrizable Kac-Moody algebra $\mfg$.  The Kostant partition functions
$K(\beta;q)$ are defined for weights $\beta$ by
\begin{equation*}
    \sum_{\beta} K(\beta;q) e^\beta = \prod_{\alpha \in \Delta^+} (1-q
        e^{\alpha})^{-\mult \alpha},
\end{equation*}
where $\Delta^+$ is the set of positive roots and $\mult \alpha = \dim
\mfg_{\alpha}$. The $q$-character of a weight space $\mcL(\lambda)_{\mu}$ is
the function
\begin{equation}\label{E:character}
    m^{\lambda}_{\mu}(q) = \sum_{w \in W} \eps(w) K(w * \lambda - \mu; q),
\end{equation}
where $W$ is the Weyl group of $\mfg$, $\eps$ is the usual sign representation
of $W$, and $w * \lambda = w(\lambda + \rho) - \rho$ is the shifted action of
$W$. The name ``$q$-character'' is used because $m^{\lambda}_{\mu}(1) = \dim
\mcL(\lambda)_{\mu}$.

When $\mfg$ is finite-dimensional it is well known that the $q$-analogs
$m^{\lambda}_{\mu}(q)$ are equal to Kostka-Foulkes polynomials, which express
the characters of highest-weight representations in terms of Hall-Littlewood
polynomials \cite{kato} \cite{gupta}, and are Kazhdan-Lusztig polynomials for
the affine Weyl group \cite{lusztig83}. When $\mu$ is dominant the coefficients
of $m^{\lambda}_{\mu}(q)$ are non-negative. There is an explanation for this
phenonemon, first conjectured by Lusztig \cite{lusztig83}: the weight
space $\mcL(\lambda)_{\mu}$ has an increasing filtration ${}^e F^*$ such that
$m^{\lambda}_{\mu}(q)$ is equal to the Poincare polynomial
\begin{equation}\label{E:poincare}
    {}^e P^{\lambda}_{\mu}(q) = \sum_{i \geq 0} q^i \dim
        {}^e F^i \mcL(\lambda)_{\mu} \ /\  {}^e F^{i-1} \mcL(\lambda)_{\mu} 
\end{equation}
of the associated graded space. This identity was first proved by Brylinski for
$\mu$ regular or $\mfg$ of classical type; the filtration ${}^e F^*$ is known
as the \emph{Brylinski} or \emph{Brylinski-Kostant} filtration, and is defined
by 
\begin{equation*}
    {}^e F^i(\mcL(\lambda)_{\mu}) = \{ v \in \mcL(\lambda)_{\mu} : e^{i+1} v = 0\}, 
\end{equation*}
where $e$ is a principal nilpotent.  Brylinski's proof was extended to all
dominant weights by Broer \cite{broer}.  More recently Joseph, Letzter, and
Zelikson gave a purely algebraic proof of the identity $m^{\lambda}_{\mu} =
{}^e P^{\lambda}_{\mu}$, and determined ${}^e P^{\lambda}_{\mu}$ for $\mu$
non-dominant \cite{jlz}. Viswanath has shown that the $q$-analogs of weight
multiplicity of an arbitrary symmetrizable Kac-Moody are Kostka-Foulkes
polynomials for generalized Hall-Littlewood polynomials, and determined
$m^{\lambda}_{\mu}(q)$ at some simple $\mu$ for an untwisted affine Kac-Moody
\cite{vis}.

The point of this paper is to extend Brylinski's result to affine (i.e.
indecomposable of affine type) Kac-Moody algebras.  We show that, as in the
finite-dimensional case, there is a filtration on $\mcL(\lambda)_{\mu}$ such
that when $\mu$ is dominant, $m^{\lambda}_{\mu}(q)$ is equal to the Poincare
series of the associated graded space. Unlike the finite-dimensional case, the
principal nilpotent is not sufficient to define the filtration in the affine
case; instead, we use the positive part of the principal Heisenberg
(this form of Brylinski's identity was first conjectured by Teleman).
Brylinski's original proof of the identity $m^{\lambda}_{\mu} ={}^e
P^{\lambda}_{\mu}$ uses a cohomology vanishing result for the flag variety. Our
proof is based on the same idea, but uses the Lie algebra cohomology approach
of \cite{fgt}. In particular we prove a vanishing result for Lie algebra
cohomology by calculating the Laplacian with respect to a Kahler metric.
Although we concentrate on the affine case for simplicity, our results
generalize easily to the case when $\mfg$ is a direct sum of algebras of finite
or affine type.  There are two difficulties in extending this result to
indefinite symmetrizable Kac-Moody algebras: there does not seem to be a simple
analogue of the Brylinski filtration, and the cohomology vanishing result does
not extend for all dominant weights $\mu$. We can overcome these difficulties
by replacing the Brylinski filtration with an intermediate filtration, and by
requiring that the root $\lambda -\mu$ has affine support. Thus we get some
partial non-negativity results for the coefficients of $m^{\lambda}_{\mu}(q)$
even when $\mfg$ is of indefinite type.

The primary motivation for this paper is a recent conjecture of Braverman and
Finkelberg. Recall that when $\mfg$ is finite-dimensional, the geometric Satake
isomorphism is an equivalence between the representation category of any group
$G$ associated to $\mfg$, and the category of equivariant perverse sheaves on
the loop Grassmannian $\Gr = G^{\vee}((z)) / G^{\vee}[[z]]$ of the Langlands
dual group $G^{\vee}$. The loop Grassmannian $\Gr$ is an ind-variety, realized
as an increasing disjoint union of Schubert varieties $\Gr^{\lambda}$
parametrized by weights of $G$. Under the equivalence, a highest-weight
representation $\mcL(\lambda)$ is sent to the intersection cohomology complex
$\IC^{\lambda}$ of $\overline{\Gr^{\lambda}}$. In addition to conjecturing the
equality $m^{\lambda}_{\mu} = {}^e P^{\lambda}_{\mu}$, Lusztig showed in
\cite{lusztig83} that $m^{\lambda}_{\mu}(q)$ is equal (after a degree shift) to
the generating function $\IC^{\lambda}_{\mu}(q)$ for the dimensions of the
stalk of the complex $\IC^{\lambda}_{\mu}$ at a point in $\Gr^{\mu} \subset
\overline{\Gr^{\lambda}}$. A direct isomorphism between the stalks
$\IC^{\lambda}_{\mu}$ and the graded spaces $\gr \mcL(\lambda)_{\mu}$ appears
in the geometric Satake isomorphism \cite{ginzburg} \cite{mv}, leading to
another proof that $m^{\lambda}_{\mu} = {}^e P^{\lambda}_{\mu}$ (see
\cite{ginzburg} in particular). Braverman and Finkelberg have proposed a
conjectural analogue of the geometric Satake isomorphism for affine Kac-Moody
groups \cite{bf}. Their conjecture relates representations of $\mfg$ to
perverse sheaves on an analogue of the loop Grassmannian for $\mfg^{\vee}$ when
$\mfg^{\vee}$ is an untwisted affine Kac-Moody. Their model leads them to
conjecture that $m^{\lambda}_{\mu}(q) = {}^e P^{\lambda}_{\mu}$ in the affine
case, with both related to the intersection cohomology stalks as in the
finite-dimensional case.\footnote{There seems to be a typo in \cite{bf}: root
multiplicities are omitted in the definition of the Kostant partition
functions.} Since we will demonstrate by example that $m^{\lambda}_{\mu}(q)$ is
not necessarily equal to ${}^e P^{\lambda}_{\mu}$, our paper gives a correction
of Braverman and Finkelberg's conjecture. 

\subsection{Acknowledgements} I thank my advisor, Constantin Teleman, for
suggesting the project and for many helpful conversations.  This work was
supported in part by NSERC. Additional support was received from NSF grants
DMS-1007255 and DMS-0709448.

\subsection{Organization} The definition of the Brylinski filtration and the
statements of the main results for affine Kac-Moody algebras are given in
Section \ref{S:bryldef}. Proofs follow in Sections \ref{S:reduction} and
\ref{S:cohvanish}. Partial results for indefinite Kac-Moody algebras are given
in Section \ref{S:indefinite}.

\subsection{Notation and terminology}\label{SS:notation}

Throughout, $\mfg$ will refer to a symmetrizable Kac-Moody algebra. For
standard notation and terminology, we mostly follow \cite{kumar}. We assume a
fixed presentation of $\mfg$, from which we get a choice of Cartan $\mfh$,
simple roots $\{\alpha_i\}$, simple coroots $\{\alpha_i^{\vee}\}$, and
Chevalley generators $\{e_i,f_i\}$. We can then grade $\mfg$ via the principal
grading, ie. by assigning degree $1$ to each $e_i$ and degree $-1$ to each
$f_i$. By choosing a real form $\mfh_{\R}$ of $\mfh$ we get an anti-linear
Cartan involution $x \mapsto \overline{x}$, defined as the anti-linear
involution sending $e_i \mapsto - f_i$ for all $i$ and $h \mapsto -h$ for all
$h \in \mfh_{\R}$. As usual $\mfg$ has the triangular decomposition $\mfg =
\overline{\mfn} \oplus \mfh \oplus \mfn$, where $\mfn$ is the standard
nilpotent $\bigoplus_{n >0} \mfg_n$. The standard Borel is the subalgebra $\mfb
= \mfh \oplus \mfn$.  Associated to $\mfn$ and $\mfb$ are the pro-algebras
$\hat{\mfn} = \lim_{\leftarrow} \mfn / \mfn_k$ and $\hat{\mfb} =
\lim_{\leftarrow} \mfb / \mfn_k$, where $\mfn_k = \bigoplus_{n > k} \mfg_n$.

The dual $\hat{\mfn}^*$ of a pro-algebra will refer to the continuous dual with
respect to the inverse limit topology. Note that as a vector space
$\hat{\mfn}^*$ is isomorphic to $\bigoplus_{n>0} \mfg_n^*$. If $V$ is a
$\hat{\mfb}$-module then $H^*_{cts}(\hat{\mfb},\mfh; V)$ will denote the
relative continuous cohomology of $(\hat{\mfb},\mfh)$ with coefficients in $V$.
Continuous cohomology is defined analogously to the ordinary cohomology, but
using continuous cochains in the Koszul complex. For complete details on
continuous cohomology we refer to \cite{Fu86}.

\section{The Brylinski filtration for affine Kac-Moody algebras}\label{S:bryldef}

A principal nilpotent (with respect to a given presentation) of a symmetrizable
Kac-Moody algebra is an element $e \in \mfg_1$ of the form $e = \sum c_i e_i$,
where $c_i \in \C \setminus \{0\}$ for all simple roots $e_i$. If $\mfg$ is
affine it is well known that the algebras $\mfs_e = \{ x \in \mfg : [x,e] \in
Z(\mfg) \}$ are Heisenberg algebras, and these algebras are called
principal Heisenberg subalgebras.

\begin{defn}
    Let $\mcL(\lambda)$ be a highest-weight module of an affine Kac-Moody
    algebra $\mfg$. Define the \emph{Brylinski filtration} with respect to the
    principal Heisenberg $\mfs$ by
    \begin{equation*}
        {}^{\mfs} F^i \mcL(\lambda)_{\mu} = \{v \in \mcL(\lambda)_{\mu} :
            x^{i+1} v = 0 \text{ for all } x \in \mfs \cap \mfn \}.
    \end{equation*} 
    Let ${}^{\mfs} P^{\lambda}_{\mu}(q)$ be the Poincare series of the
    associated graded space of $\mcL(\lambda)_{\mu}$. 
\end{defn}
Note that the principal nilpotents form a single $H$-orbit, so the filtration
${}^{\mfs} F^*$ is independent of the choice of principal Heisenberg. 

Recall that a weight $\mu$ is real-valued if $\mu(h) \in \R$ for all $h \in
\mfh_{\R}$, and dominant if $\mu(\alpha_i^{\vee}) \geq 0$ for all simple
coroots $\alpha_i^{\vee}$. 
\begin{theorem}\label{T:brylpoincare}
    Let $\mcL(\lambda)$ be an integrable highest weight representation of an
    affine Kac-Moody algebra $\mfg$, where $\lambda$ is a real-valued dominant
    weight.  If $\mu$ is a dominant weight of $\mcL(\lambda)$ then
    ${}^{\mfs} P^{\lambda}_{\mu}(q) = m^{\lambda}_{\mu}(q)$.  
\end{theorem}

As in Subsection \ref{SS:notation}, $\hat{\mfn}^*$ denotes the continuous dual
and $H^*_{cts}$ denotes continuous Lie algebra cohomology.  The proof of
Theorem \ref{T:brylpoincare} depends on 
\begin{thm}\label{T:brylkm} Let $\mcL(\lambda)$ be an integrable highest weight
    representation of an affine Kac-Moody algebra $\mfg$, where $\lambda$ is a
    real-valued dominant weight. Let $V = \mcL(\lambda) \otimes S^* \hat{\mfn}^*
    \otimes \C_{-\mu}$, where $\mu$ is a dominant weight of $\mcL(\lambda)$.
    Then $H^d_{cts}(\hat{\mfb},\mfh; V) = 0$ for $d>0$, and in addition there
    is a graded isomorphism $\gr \mcL(\lambda)_{\mu} \iso
    H^0_{cts}(\hat{\mfb},\mfh; V)$, where the latter space is graded by
    symmetric degree.  
\end{thm}

\begin{proof}[Proof of Theorem \ref{T:brylpoincare} from Theorem \ref{T:brylkm}]

Let $V^p = \mcL(\lambda) \otimes S^p \hat{\mfn}^* \otimes \C_{-\mu}$.  By
Theorem \ref{T:brylkm}, $P^{\lambda}_{\mu}(q) = \sum_{p \geq 0} \dim
H^0_{cts}(\hat{\mfb},\mfh; V^p) q^p = \sum \chi(\hat{\mfb},\mfh; V^p) q^p$,
where $\chi$ is the Euler characteristic (the second equality follows from
cohomology vanishing). Since $\hat{\mfn}^*$ has finite-dimensional weight
spaces and all weights belong to the negative root cone, $\bigwedge^*
\hat{\mfn}^* \otimes \mcL(\lambda) \otimes S^p \hat{\mfn}^*$ has
finite-dimensional weight spaces. Thus we can write
\begin{align*}
    \sum_{p \geq 0} \chi(\hat{\mfb},\mfh; V^p) q^p & = \sum_{p,k \geq 0} (-1)^k
        q^p \dim \left(\bigwedge^k \hat{\mfn}^* \otimes V^p\right)^{\mfh} \\
            & =[e^{\mu}] \ch \mcL(\lambda) \prod_{\alpha \in \Delta^+}
                (1-e^{-\alpha})^{\mult \alpha} (1 - q e^{-\alpha})^{-\mult\alpha}.
\end{align*}
Applying the Weyl-Kac character formula 
\begin{equation*}
    \ch \mcL(\lambda) = \sum_{w \in W} \eps(w) e^{w * \lambda} \cdot
        \prod_{\alpha \in \Delta^+} (1 - e^{-\alpha})^{- \mult \alpha}
\end{equation*}
we get the result.

\end{proof}
The proof of Theorem \ref{T:brylkm} will be given in Sections \ref{S:reduction}
and \ref{S:cohvanish}. If $\mfg = \bigoplus \mfg_i$ is a direct sum of
indecomposables of finite and affine type, the conclusions of Theorems
\ref{T:brylpoincare} and \ref{T:brylkm} remain true with $\mfs$ replaced by a
direct sum of principal nilpotents (for the finite components) and principal
Heisenbergs (for the affine components). 

\subsection{Examples}

We now give some elementary examples to show that ${}^{\mfs} F$ is different
from ${}^e F$. Consider $\widehat{\mfsl_2}$, the affine Kac-Moody algebra
realized as $\mfsl_2[z^{\pm 1}] \oplus \C c \oplus \C d$, where $c$ is a
central element, and $d$ is the derivation $\pd{}{z}$. Let $\{H,E,F\}$ be an
$\mfsl_2$-triple in $\mfsl_2$, and take principal nilpotent $e = E + Fz$. The
principal Heisenberg $\mfs$ is spanned by the elements $e z^n$, $n \in \Z$,
along with $c$.

The Cartan subalgebra of $\widehat{\mfsl_2}$ is $\vspan \{H,c,d\}$. Denote a
weight $\alpha H^* + h c^* + n d^*$ by $(\alpha,h,n)$. The weight $\lambda =
(\alpha,h,n)$ is dominant if $0 \leq \alpha \leq h$, and the corresponding
irreducible highest-weight representation $L(\lambda)$ can be realized as the
quotient of the Verma module $U(\mfg) \otimes_{U(\mfb)} \C_{\lambda}$ by the
$U(\mfg)$-submodule generated by $F^{\alpha+1} \otimes 1$ and $(E z^{-1})^{h -
\alpha + 1} \otimes 1$. Let
\begin{equation*}
    w = (F z^{-1}) (E z^{-1}) v,
\end{equation*}
where $v$ is the highest weight vector in $L(c^*)$. Note that $w$ is a weight
vector of weight $(0,1,-2)$. It is easy to check, using the defining relations
for $L(c^*)$, that $e^2 w = 0$, while $(e z) e w = 3 v$, so $w \in {}^e F^2$
but is not in ${}^{\mfs} F^2$.

The same idea can be used to calculate Poincare series. For the above example,
where $\lambda = (0,1,0)$ and $\mu = (0,1,-2)$, we have $\dim
\mcL(\lambda)_{\mu} = 2$. The Poincare series for ${}^e F$ is $q + q^4$, while
the Poincare series for ${}^{\mfs} F$ is $m^{\lambda}_{\mu}(q) =  q^2 + q^4$.
For an example with a dominant regular weight, let $\lambda = (0,3,0)$ and
$\mu = (2,3,-3)$. The Poincare series of ${}^e F$ is $q + 2q^2 + q^3 + q^5$,
while $m^{\lambda}_{\mu}(q) = q + q^2 + 2q^3 + q^5$.

\section{Reduction to cohomology vanishing}\label{S:reduction}

In this section we introduce an equivalent filtration to the Brylinski
filtration, which will allow us to reduce Theorem \ref{T:brylkm} to a
cohomology vanishing statement. The line of argument is inspired by \cite{bryl}
and \cite{fgt}. As usual, $\mfg$ will be an arbitrary symmetrizable Kac-Moody
algebra except where stated.

Associated to $\mfg$ is a Kac-Moody group $\mcG$. The standard Borel subgroup
$\mcB$ of $\mcG$ is a solvable pro-group with Lie algebra $\hat{\mfb}$. The
standard unipotent subgroup $\mcU \subset \mcB$ is a unipotent pro-group with
Lie algebra $\hat{\mfn}$. The Borel $\mcB$ also contains a torus $H$ corresponding
to $\mfh$. Defining the new filtration requires two lemmas. 
\begin{lemma}\label{L:affinestruct}
    There are algebraic isomorphisms $\mcU \iso \mcB / H \iso \hat{\mfn}$
    giving $\mcU$ the structure of a linear space with an affine $\mcB$-action.
\end{lemma}
\begin{proof}
    Note that the spaces in question can be naturally expressed as inverse
    limits of affine schemes, and hence are affine schemes in their own right.
    Pick $\delta \in \mfh$ acting on $\mfg_n$ as multiplication by $n$, and
    define $\pi : \mcB \arr \hat{\mfn}$ by $\Ad(b) \delta = \delta + \pi(b)$.
    Then the composition $\mcU \incl \mcB \arr \mcB / H \arr \hat{\mfn}$ is an
    isomorphism. $\hat{\mfn}$ has a linear structure, while $\mcB / H$ has a
    left-translation action of $\mcB$. If $b_1,b_2 \in \mcB$ then $\Ad(b_1 b_2)
    \delta = \Ad(b_1) (\delta + \pi(b_2)) = \delta + \pi(b_1) + \Ad(b_1)
    \pi(b_2)$, so $\pi(b_1 b_2) = \Ad(b_1) \pi(b_2) + \pi(b_1)$ and the
    resulting action of $\mcB$ on $\hat{\mfn}$ is affine. 
\end{proof}
\begin{lemma}\label{L:filteredinvariants}
    Let $V$ be a pro-representation of $\mcB$. Then evaluation at the identity
    gives an isomorphism $(V \otimes \C[\mcU])^{\mcB} \arr V^H$.
\end{lemma}
\begin{proof}
    Any element $v \in V^H$ extends to a $\mcB$-invariant function $\mcU
    \arr V$ by $[b] \mapsto b v$.
\end{proof}
The linear structure on $\hat{\mfn}$ and the isomorphism of $\mcU$ with
$\hat{\mfn}$ gives a $\mcB$-stable filtration of $\C[\mcU]$ by polynomial
degree.  Lemma \ref{L:filteredinvariants} implies that if $V$ is a
pro-representation of $\mcB$ then $V^H$ can be filtered via polynomial degree
on $\C[\mcU]$. If $\mu$ is a weight of $\mfg$ then extending $\mu$ by zero on
$\mcU$ makes $\C_{-\mu}$ into a pro-representation of $\mcB$.  The reason for
introducing a new filtration is the following lemma, which reduces the proof of
Theorem \ref{T:brylkm} to a vanishing result.
\begin{lemma} Let $W = \mcL(\lambda) \otimes \C_{-\mu}$, and filter
    $\mcL(\lambda)_{\mu} = W^H$ via the isomorphism $W^H \iso (W \otimes
    \C[\mcU])^{\mcB}$. If $H^1_{cts}(\hat{\mfb},\mfh; W \otimes S^*
    \hat{\mfn}^*) = 0$ then $H^0_{cts}(\hat{\mfb},\mfh; W \otimes S^*
    \hat{\mfn}^*) \iso \gr \mcL(\lambda)_{\mu}$.
\end{lemma}
\begin{proof}
    Let $\mcF^p$ be the subset of $\C[\mcU]$ of polynomials of degree at most
    $p$.  Then $\gr \C[\mcU] = S^* \hat{\mfn}^*$ as $\mcB$-modules, so there
    are short exact sequences
    \begin{equation*}
        0 \arr W\otimes \mcF^{p-1} \arr W \otimes \mcF^{p} \arr W\otimes S^p
            \hat{\mfn}^* \arr 0 
    \end{equation*}
    of $\mcB$-modules for all $p$. The corresponding long exact sequence in Lie
    algebra cohomology is
    \begin{align*}
         H^i_{cts}(\hat{\mfb},\mfh;W\otimes \mcF^{p-1}) \arr 
          H^i_{cts}(\hat{\mfb},\mfh;W \otimes \mcF^{p}) \arr &
          H^i_{cts}(\hat{\mfb},\mfh; W\otimes S^p \hat{\mfn}^* ) \\
            & \arr H^{i+1}_{cts}(\hat{\mfb},\mfh; W\otimes \mcF^{p-1}).
    \end{align*}
    Since $H^i_{cts}(\hat{\mfb},\mfh; W\otimes S^p \hat{\mfn}^* ) = 0$ for
    $i=1$, the inclusion $W \otimes \mcF^{p-1} \incl W\otimes \mcF^{p}$ induces
    a surjection in degree one cohomology for all $p$. Since $\mcF^{-1} = 0$,
    $H^1_{cts}(\hat{\mfb},\mfh;W \otimes \mcF^{p}) =0$ for all $p$. The long
    exact sequence in degree $i=0$ gives an isomorphism
    $H^0_{cts}(\hat{\mfb},\mfh; W \otimes S^p \hat{\mfn}^* ) \iso (W \otimes
    \mcF^{p})^{\mfb} / (W\otimes \mcF^{p-1})^{\mfb}$. This latter quotient is
    the graded space of $(W\otimes \C[U])^{\mfb}$ as required.
\end{proof}

Now we show that the new filtration is equal to the Brylinski filtration when
$\mfg$ is affine.
\begin{prop}\label{P:heisenberg}
    Let $\mcL(\lambda)$ be an integrable highest-weight representation of an
    affine Kac-Moody $\mfg$. Then the Brylinski filtration on a weight space
    $\mcL(\lambda)_{\mu}$ agrees with the filtration of $\mcL(\lambda)_{\mu}
    \iso (\mcL(\lambda) \otimes \C_{-\mu} \otimes \C[\mcU])^{\mcB}$ by
    polynomial degree.
\end{prop}

The proof of Proposition \ref{P:heisenberg} requires two lemmas.
\begin{lemma}
    If $\mfg$ is affine and $\mfs$ is a principal Heisenberg then $\Ad(\mcB)
    (\mfs \cap \mfn)$ is dense in $\hat{\mfn}$.
\end{lemma}
\begin{proof}
    The principal nilpotents form a single orbit, so it is only necessary to
    prove this fact for a single principal nilpotent. We claim that there is a
    principal nilpotent such that $f = -\overline{e} \in \mfs_e$, so that in
    particular $[e,f] \in Z(\mfg)$. Indeed, let $A$ be the generalized
    Cartan matrix defining $\mfg$, i.e.  $A_{ij} = \alpha_j(\alpha_i^{\vee})$.
    Since $\mfg$ is affine there is a vector $c > 0$, unique up to a scalar
    multiple, such that $A^t c = 0$. If we pick $e = \sum \sqrt{c_i} e_i$ then
    $[e,f] = \sum c_i \alpha_i^{\vee}$, and $\alpha_j([e,f]) = \sum c_i A_{ij}
    = (A^t c)_j = 0$ for all simple roots $\alpha_j$. 

    Now we show that $\mfn = (\mfs_e \cap \mfn) + [\mfb,e]$. In degree one we
    have $[\mfh,e] = \mfg_1$. For higher degrees, let $\{,\}$ denote the
    standard non-degenerate contragradient Hermitian form on $\mfg$ which is
    positive definite on $\mfn$.  An element $x \in \mfn$ is orthogonal to
    $[\mfb,e]$ if and only if $0 = \{[e,z],x\} = \{z,[f,x]\}$ for all $z \in
    \mfb$, or in other words if and only if $x \in C_{\mfg}(f)$. Suppose $x \in
    \mfg_n$, $n \geq 2$ belongs to $[\mfb,e]^{\perp}$.  Using the fact that $[e,f]
    \in Z(\mfg)$ we get that $\{[e,x],[e,x]\} = \{[f,x],[f,x]\} = 0$, and
    conclude that $x \in \mfs_e$.

    $(\mfs \cap \mfn) + [\mfb,e] = \mfn$ implies that the multiplication map
    $\mcB \times (\mfs \cap \mfn) \arr \hat{\mfn}$ is a submersion in a
    neighbourhood of $(\Id,e)$.  Since $\mcB$ acts algebraically on $\mfs \cap
    \mfn \subset \hat{\mfb}$, the subset $\mcB (\mfs \cap \mfn)$ is dense in
    $\hat{\mfn}$.
\end{proof}

\begin{lemma}
    Let $\mcL(\lambda)$ be an integrable highest-weight module. Considered as a
    $\mcB$-module, $\mcL(\lambda)$ is a submodule of $\C[\mcU] \otimes
    \C_{\lambda}$.
\end{lemma}
\begin{proof}
    This statement would follow immediately from a Borel-Weil theorem for the
    thick flag variety of a Kac-Moody group. As we are not aware of a formal
    statement of the Borel-Weil theorem in this context, we recover the result
    from the dual of the quotient map $M_{low}(-\lambda) \arr
    \mcL_{low}(-\lambda)$, where $M_{low}(-\lambda) = U(\mfg)
    \otimes_{U(\overline{\mfb})} \C_{-\lambda}$ is a lowest weight Verma module,
    and $\mcL_{low}(-\lambda)$ is the irreducible representation with lowest
    weight $-\lambda$. Both these spaces are $\mfg$-modules with finite
    gradings induced by the principal grading of $\mfg$. Let
    $M_{low}(-\lambda)^*$ and $\mcL(-\lambda)^*$ denote the finitely-supported
    duals, consisting of linear functions which are supported on a finite
    number of graded components. 

    Using the fact that $M_{low}(-\lambda)$ is a free $U(\mfn)$-module, we can
    identity $M_{low}(-\lambda)$ with $S^* \mfn \otimes \C_{-\lambda}$ where
    $S^* \mfn$ has the $\mfb$-action $(y,x) \mapsto [y,\delta] \circ x +
    \ad(y)x$, the symbol $\circ$ denotes symmetric multiplication, and $\delta$
    is defined as in Lemma \ref{L:affinestruct} as an element of $\mfh$ which
    acts on $\mfg_n$ as multiplication by $n$.  The finitely supported dual of
    $M_{low}(-\lambda)$ can be identified with $S^* \hat{\mfn}^* \otimes
    \C_{\lambda}$ where $\mfb$ acts on $S^* \hat{\mfn}^*$ by $(y,f) \mapsto
    \ad^t(y) f + \iota([\delta,y]) f$.  It is not hard to check that this
    action integrates to the $\mcB$-action coming from identifying $S^*
    \hat{\mfn}^*$ with $\C[\mcU]$. Since the quotient map preserves the
    principal grading, the dual of the surjection $M_{low}(-\lambda) \arr
    \mcL_{low}(-\lambda)$ is an inclusion $\mcL(\lambda) =
    \mcL_{low}(-\lambda)^* \incl M_{low}(-\lambda)^* = \C[\mcU] \otimes
    \C_{\lambda}$ as required. 
\end{proof}

\begin{proof}[Proof of Proposition \ref{P:heisenberg}]
    Let $V = \C_{\beta} \otimes \C[\mcU]$, where $\beta = \lambda - \mu$. By
    the last lemma, we can prove the proposition with $\mcL(\lambda)_{\mu}$
    replaced by $V^H$, where the
    filtration on $V^H$ is defined by $V^H \iso (V \otimes \C[\mcU])^{\mcB}$.
    An element $f$ of this latter set can be identified with a $\mcB$-invariant
    function $\mcU \times \mcU \arr \C_{\beta}$. The polynomial degree on the
    second factor is the maximum $t$-degree of $f(u,tx)$ as $u$ ranges across
    $\mcU$ and $x$ ranges across $\hat{\mfn} \iso \mcU$. Suppose this maximum
    is achieved at $(u_0,x_0)$.  Since $\mcB (\mfs \cap \mfn)$ is dense in
    $\hat{\mfn}$, we can assume that $x_0 = \Ad(b) s$ for $b \in \mcB$ and $s
    \in \mfs \cap \mfn$.  Now $\mfs \cap \mfn$ is abelian and graded, so the
    graded components of $s$ commute with each other. This allows us to find
    $\tilde{s} \in \mfs \cap \mfn$ such that $\pi(e^{t \tilde{s}}) = t s$.
    Since the degree of $f(u_0,\cdot)$ is achieved on the line $\Ad(b) \pi(e^{t
    \tilde{s}})$, it is also achieved on the parallel line $\Ad(b) \pi(e^{t
    \tilde{s}}) + \pi(b) = \pi(b e^{t \tilde{s}})$. Thus the polynomial degree
    of $f$ is equal to the $t$-degree of $f(u_0,b \pi(e^{t \tilde{s}})) =
    \beta(b) f(b^{-1} u_0, \pi(e^{t \tilde{s}}))$. Since $\beta(b)$ is a
    non-zero scalar, we conclude that there is $u \in \mcU$ and $s \in \mfs
    \cap \mfn$ such that the degree of $f$ is equal to the $t$-degree of
    $f(u,\pi(e^{t s}))$.  Conversely if $s \in \mfs\cap \mfn$ then $\pi(e^{t
    s})$ is a line in $\hat{\mfn}$, so the degree of $f$ is equal to the
    $t$-degree of $f(u,\pi(e^{t s}))$ as $u$ ranges across $\mcU$ and $s$
    ranges across $\mfs \cap \mfn$.

    Given $f \in (\C_{\beta} \otimes \C[\mcU] \otimes \C[\mcU])$ let $\tilde{f}
    \in \C_{\beta} \otimes \C[\mcU]$ be the restriction to $\mcU \times \{\Id\}$.
    The $\mcB$-action on $\C_{\beta} \otimes \C[\mcU]$ is defined by $(b\cdot f)(u)
    = \beta(b) f(b^{-1} u)$, so if $f$ is $\mcB$-invariant then the $t$-degree
    of $f(u,\pi(e^{ts}))$ is equal to the $t$-degree of $(e^{ts}
    \tilde{f})(u)$. Since
    \begin{equation*}
        e^{ts} \tilde{f} = \sum_{n \geq 0} \frac{t^n}{n!} s^n \tilde{f},
    \end{equation*}
    the degree of $f$ is equal to the smallest $n$ such that $s^{n+1}
    \tilde{f} = 0$ for all $s \in \mfs\cap \mfn$.
\end{proof}
The proof of Proposition \ref{P:heisenberg} works just as well with $\mfs \cap
\mfn$ replaced by any graded abelian subalgebra $\mfa$ of $\hat{\mfn}$ such
that $\Ad(\mcB) \mfa$ is dense in $\hat{\mfn}$. For example, in the
finite-dimensional case we could take $\mfa = \C e$. If $\mfg = \bigoplus
\mfg_i$ is a direct sum of indecomposables of finite or affine type then
we can take $\mfa = \bigoplus \mfa_i$, where $\mfa_i$ is either the positive
part of the principal Heisenberg, or the line through the positive nilpotent,
depending on whether $\mfg_i$ is affine or finite.

\section{Cohomology vanishing}\label{S:cohvanish}

\subsection{Nakano's identity and the Laplacian}\label{SS:nakano}

We need some tools to prove the necessary cohomology vanishing result.
Throughout this section $\mfg$ will be an arbitrary symmetrizable Kac-Moody
algebra.  $(V,\pi)$ will be a $\hat{\mfb}$-module such that $\pi|_{\mfg_0}$
extends to an action of $\overline{\mfb}$ (this conjugate action will also be
denoted by $\pi$). Note that since $\mfn = \mfg / \overline{\mfb}$,
$\hat{\mfn}^*$ is both a $\hat{\mfb}$-module and a $\overline{\mfb}$-module.
The space $\overline{\mfn} = \mfg / \mfb$ has the same property. 
\begin{defn}
    The \emph{semi-infinite chain complex} $(C^{*,*}(V),\deltabar, D)$
    is the bicomplex 
    \begin{equation*}
        C^{-a,b}(V) = \left(\bigwedge^{b} \hat{\mfn}^* \otimes \bigwedge^a
            \overline{\mfn} \otimes V\right)^{\mfg_0}.
    \end{equation*}
    with differentials $\deltabar$ and $D$, where the former is the Lie algebra
    cohomology differential of $\hat{\mfn}$ with coefficients in $\bigwedge^*
    \overline{\mfn} \otimes V$, and the latter is the Lie algebra homology
    differential of $\overline{\mfn}$ with coefficients in $\bigwedge^* \hat{\mfn}^*
    \otimes V$, both restricted to $\mfg_0$-invariants. 
\end{defn}
To make the definition of $\deltabar$ and $D$ more explicit, identify
$C^{*,*}(V)$ with $\bigwedge^* (\hat{\mfn}^* \oplus \bar{\mfn}) \otimes V$.
Then the Clifford algebra of $\mfn \oplus \bar{\mfn} \oplus \hat{\mfn}^* \oplus
\hat{\bar{\mfn}}^*$ with the dual pairing acts on $C^{*,*}(V)$, where
$\hat{\mfn}^*$ and $\bar{\mfn}$ act by exterior multiplication, and $\mfn$
and $\hat{\bar{\mfn}}^*$ act by interior multiplication. Pick a homogeneous
basis $\{z_i\}_{i \geq 1}$ for $\mfn$, let $\{z^i\}$ denote the dual basis,
and let $z_{-i} = \overline{z_i}$. Then 
\begin{equation*}
    \deltabar = \sum_{k \geq 1} \eps(z^k) \left(\frac{1}{2} \ad^t_n(z_k) + 
        \ad_{\overline{\mfn}} (z_k) + \pi(z_k) \right),
\end{equation*}
where $\eps$ is exterior multiplication, while
\begin{equation*}
    D = \sum_{k \geq 1} \left( \frac{1}{2} \ad_{\overline{\mfn}} (z_{-k})
        + \ad_{\mfn}^t(z_{-k}) + \pi(z_{-k}) \right) \iota(z^{-k}),
\end{equation*}
where $\iota$ is interior multiplication.
    
The semi-infinite cocycle is defined by $\gamma|_{\mfg_m \times \mfg_n} = 0$ if
$m + n \neq 0$ and by 
\begin{equation*}
    \gamma(x,y) = \sum_{0 \leq n < k} \tr_{\mfg_n}(\ad(x)\ad(y))
\end{equation*}
for $x \in \mfg_k$, $y \in \mfg_{-k}$, $k \geq 0$. Since $\mfh = \mfg_0$ is
abelian, $(x,y) = -\gamma(x,\overline{y})$ defines a Hermitian form on $\mfn$.
\begin{lemma}\label{L:kahler}
    Let $\langle,\rangle$ be a symmetric invariant form on $\mfg$
    (real-valued on a real-form of $\mfg$) such that $\{\cdot,\cdot\} =
    -\langle \cdot, \overline{\cdot} \rangle$ is contragradient and
    positive-definite on $\mfn$. Then the Hermitian form $(\cdot,\cdot) =
    -\gamma(\cdot,\overline{\cdot})$ on $\mfn$ agrees with the form defined by
    \begin{equation*}
        (x,y) = 2 \langle \rho, \alpha \rangle \{x,y\}, x \in \mfg_{\alpha}.
    \end{equation*}
\end{lemma}
\begin{proof}
    Suppose $x,y \in \mfg_\alpha$. If $\{u_i\}$ and $\{u^i\}$ are dual bases of
    $\mfh$ with respect to $\langle,\rangle$ then 
    \begin{align*}
        \tr_{\mfg_0}(\ad(x) \ad(\overline{y})) & = \sum_i \langle u_i, 
                        [x,[\overline{y},u^i]] \rangle\\ 
            &= \langle x, \overline{y} \rangle \langle \alpha,\alpha \rangle.
    \end{align*}
    Next, let $\{e^i_{\beta}\}$ and $\{e^i_{-\beta}\}$ be dual bases of $\mfg_\beta$
    and $\mfg_{-\beta}$ with respect to $\langle,\rangle$. Let $\rho \in
    \mfh^*$ be such that $\rho(\alpha_i^{\vee}) = 1$ for all coroots
    $\alpha_i^{\vee}$. Then
    \begin{align*}
        \gamma(x,y) & = \langle x, \overline{y} \rangle \langle \alpha,\alpha \rangle
            +  \sum_{\beta \in \Delta^+} \sum_i \langle e^i_{-\beta},
                [x,[\overline{y}, e^i_{\beta}]_{-}] \rangle,
    \end{align*}
    where $x_{-}$ is the projection of $x \in \mfg$ to $\overline{\mfn}$ using the
    triangular decomposition. Rearranging $\langle e^i_{-\beta}, [x,[\overline{y},
    e^i_{\beta}]_{-}] \rangle = \langle x,
    [e^i_{-\beta},[e^i_{\beta},\overline{y}]_{-}] \rangle$ and applying Lemma
    2.3.11 of \cite{kumar}, we get that $\gamma(x,\overline{y}) = 2\langle
    \rho, \alpha \rangle \langle x, \overline{y} \rangle$. 
\end{proof}
The result of Lemma \ref{L:kahler} is that $(,)$ defines a
$\mfg_0$-contragradient Kahler metric on $\mfn$. Suppose $V$ has a
positive-definite Hermitian form contragradient with respect to $\pi$. Using
the Kahler metric on $\mfn$, we can give $C^{*,*}(V)$ a positive-definite
Hermitian form by defining $(\overline{x},\overline{y}) = \overline{(x,y)}$ for
$x,y \in \mfn$. Let $\boxbar = \deltabar \deltabar^* + \deltabar^* \deltabar$
be the $\deltabar$-Laplacian, and $\square = D D^* + D^* D$ be the
$D$-Laplacian.  Then a version of Nakano's ipentity holds:
\begin{prop}[Nakano's identity, see \cite{tel}]\label{P:nakano} 
    The $\deltabar$-Laplacian $\boxbar$ and the $D$-Laplacian $\square$ are
    related by
    \begin{equation*}
        \boxbar = \square + \degree + \curv,
    \end{equation*}
    where $\degree$ acts on $C^{a,b}(V)$ as multiplication by $a+b$, and
    \begin{equation*}
        \curv = -\sum_{i,j \geq 1} \eps(z^i) \iota(z_j) \left( [\pi(z_i),\pi(z_{-j})] -
            \pi([z_i,z_{-j}]) \right),
    \end{equation*}
    on $C^{0,b}(V)$ for $\{z_i\}$ a homogeneous basis of $\mfn$ orthonormal in
    $(,)$.  
\end{prop}

\subsection{Laplacian calculation for symmetrizable Kac-Moody algebras}

Given an operator $T$ on $\hat{\mfn}^*$, let $d_R(T)$ and $d_L(T)$ denote the
operators on $\bigwedge^* \hat{\mfn}^* \otimes S^* \hat{\mfn}^*$ defined by
\begin{equation*}
    \alpha_1 \wedge \ldots \wedge \alpha_k
        \otimes \beta \mapsto
    \sum_{i=1}^k (-1)^i \alpha_1 \wedge \ldots \check{\alpha}_i \ldots
        \wedge \alpha_k
        \otimes T(\alpha_i) \circ \beta
\end{equation*}
and
\begin{equation*}
    \alpha \otimes \beta_1 \circ \ldots \circ \beta_l
        \mapsto \sum_{i=1}^l T(\beta_i) \wedge \alpha \otimes \beta_1 \circ \ldots
            \circ \check{\beta}_i \circ \ldots \circ \beta_l
\end{equation*}
respectively.  Define an operator $J$ on $\hat{\mfn}^*$ by $f \mapsto f / 2
\langle \rho, \alpha \rangle$ if $f \in \mfg_{\alpha}^*$. As in the last
section, let $\langle, \rangle$ be a real-valued symmetric invariant bilinear
form such that $\{,\} = -\langle \cdot, \overline{\cdot} \rangle$ is
contragradient and positive-definite on $\mfn$.
\begin{prop}\label{P:laplacian}
    Extend the contragradient Hermitian form $\{,\}$ on $\mfn$ to $V = S^*
    \hat{\mfn}^*$. On $C^{0,b}(V)$, 
    \begin{equation*}
        \curv_V = \sum_{s \geq 0} d_L(\ad^t(y_s')) d_R(\ad^t(y_s) J) - \deg,
    \end{equation*}
    where $\{y_s\}$ is a homogeneous basis for $\mfb$ and $\{y_s'\}$ is a basis
    for $\overline{\mfb}$ dual with respect to $\langle,\rangle$.
\end{prop}
\begin{proof}
    Let $V' = S^* \overline{\mfn}$, and let $\pi$ denote the actions of
    $\mfb$ and $\overline{\mfb}$ on $V'$. From Proposition \ref{P:nakano} we
    see that $\curv_{V'}$ is a second-order differential operator, and thus is
    determined by its action on $\hat{\mfn}^* \otimes \overline{\mfn}$.  We
    claim that if $f \in \hat{\mfn}^*$ and $w \in \overline{\mfn}$ then
    \begin{equation*}
        \curv_{V'}(f \otimes w) = \sum_{s \geq 0} \ad^t_{\mfn}(w) y^s
            \otimes \ad_{\overline{\mfn}}(y_s) \phi^{-1}(f),
    \end{equation*}
    where $\phi : \overline{\mfn} \arr \hat{\mfn}^*$ is the isomorphism induced
    by the Kahler metric, and $\{y_s\}$ is any homogeneous basis of $\mfb$.
    To prove this claim, let $\{z_i\}$ be orthonormal with respect to the
    Kahler metric, and think about $f = z^k$, $w = z_{-l}$. Observe that
    \begin{equation*}
        \pi(z) w = \sum_{i < 0} z^i([z,w]) z_i.
    \end{equation*}
    Using this expression, we get that if $z_{-j} \in \mfg_{-m}$ then 
    \begin{equation*}
        ([\pi(z_i),\pi(z_{-j})] - \pi([z_i,z_{-j}]))w = \sum_{-m \leq n < 0} \sum_{z_{-k} \in \mfg_n}
            z^{-k} ([z_{-j},[z_i,w]]) z_{-k}.
    \end{equation*}
    We can then remove the reference to $m$ and write
    \begin{equation*}
        ([\pi(z_i),\pi(z_{-j})] - \pi([z_i,z_{-j}]))w = \sum_{k > 0} \sum_{s \geq 0}
                z^{-k}([z_{-j},y_s]) y^s([z_i,w]) z_{-k}.
    \end{equation*}
    Now we can calculate
    \begin{align*}
        \curv_{V'}(z^k \otimes z_{-l}) & 
         = -\sum_{i>0} z^i \otimes \left([\pi(z_i),\pi(z_{-k})] - \pi([z_i,z_{-k}])\right) z_{-l} \\
      & = -\sum_{i,j>0} \sum_{s \geq 0} z^i \otimes z^{-j}([z_{-k},y_s]) y^s([z_i,z_{-l}]) z_{-j}. \\
    \end{align*}            
    By summing over $z_i \in \mfg_n$ for fixed $n$, it is possible to move the
    $z_{-l}$ action from $z_i$ to $z^i$. The last expression becomes
    \begin{equation*}
        -\sum_{s \geq 0} \sum_{j>0} (\ad^t(z_{-l}) y^s) \otimes z^{-j}([z_{-k},y_s]) z_{-j}
        = \sum_{s \geq 0} (\ad^t(z_{-l}) y^s) \otimes \pi(y_s)(z_{-k}).
    \end{equation*}
    The proof of the claim is finished by noting that $z_{-k} = \phi^{-1}(z^k)$.

    Next, the contragradient metric $\{,\}$ gives an isomorphism $\psi :
    \overline{\mfn} \arr \hat{\mfn}^*$ of $\mfb$ and $\overline{\mfb}$-modules.
    $J = \psi \phi^{-1}$, while $\ad^t(w) y^s = \ad^t(y_s') \psi(w)$ where
    $\{y_s'\}$ is the dual basis to $\{y_s\}$. Identifying $V$ with $V'$ via
    $\psi$ gives
    \begin{equation*}
        \curv_{V}(f \otimes g) = \sum_{s \geq 0} \ad^t(y_s') g \otimes \ad^t(y_s) J f.
    \end{equation*}

    Given $S,T \in \End(\hat{\mfn}^*)$, define a second-order operator $\Switch(S,T)$
    on $\bigwedge^* \hat{\mfn}^* \otimes S^* \hat{\mfn}^*$ by $f \otimes g
    \mapsto T g \otimes S f$. Then $\Switch(S,T) = d_L(T) d_R(S) - (TS)^{\wedge}$,
    where $(TS)^{\wedge}$ is the extension of $TS$ to $\bigwedge^*
    \hat{\mfn}^*$ as a derivation. We have shown that
    \begin{equation*}
        \curv_{V} = \sum_{s \geq 0} \Switch(\ad^t(y_s)J,\ad^t(y_s'))
            = \sum_{s \geq 0} d_L(\ad^t(y_s')) d_R(\ad^t(y_s) J) - (T J)^{\wedge},
    \end{equation*}
    where $T = \sum_{s \geq 0} \ad^t(y_s') \ad^t(y_s)$. It is not hard to see
    that $(T \psi(y))(x) = -\gamma(x,y)$ for $x \in \mfn$, $y \in
    \overline{\mfn}$, so $T = J^{-1}$ by Lemma \ref{L:kahler}.
\end{proof}
Note that $d_R(T J) = d_L(T^*)$, where $T^*$ is the adjoint of $T \in
\End(\hat{\mfn}^*)$ in the contragradient metric. The map $J$ appears because
the Kahler metric is used on $\bigwedge^* \hat{\mfn}^*$ while the contragradient
metric is used on $S^* \hat{\mfn}^*$.  Since the isomorphism $\psi$ appearing
in the proof is an isometry, $\ad^t(x)^* = -\ad(\overline{x})^*$ in the 
contragradient metric.

\subsection{Cohomology vanishing for affine Kac-Moody algebras}\label{SS:cohvanish}

If $\mfg$ is affine then $\mfg$ can be realized as the algebra $(L[z^{\pm 1}]
\oplus \C c \oplus \C d)^{\tilde{\sigma}}$, where $L$ is a simple Lie algebra
and $\tilde{\sigma}$ is an automorphism of $\mfg$ defined by
\begin{equation*}
    \tilde{\sigma}(c) = c, \tilde{\sigma}(d) = d, \tilde{\sigma}(x z^n)
        = \zeta^{-n} \sigma(x) z^n, \quad x \in L
\end{equation*}
for $\sigma$ a diagram automorphism of $L$ of finite order $k$ and $\zeta$ a
fixed $k$th root of unity. We use the conventions of \cite{Ka83} (see chapters
7 and 8 in particular). The bracket is defined by
\begin{align*}
    [x z^m + \gamma_1 c + \beta_1 d, y z^n + \gamma_2 c + \beta_2 d]
        = & \\
       [x,y] z^{m+n} + & \beta_1 n y z^{n} - \beta_2 m x z^{m} +
            \delta_{m,-n} m \langle x, y \rangle c,
\end{align*}
for $x,y \in L$, $\gamma_1,\gamma_2,\beta_1,\beta_2 \in \C$, where
$\langle,\rangle$ is the symmetric invariant bilinear form on $L$ normalized
by setting the length squared of a long root to $2 k$.  The diagram automorphism
acts diagonalizably on $L$, so that
\begin{equation*}
    \mfg = \bigoplus_{i = 0}^{k-1} L_i z^i \otimes \C[z^{\pm k}] \oplus 
        \C c \oplus \C d,
\end{equation*}
where $L_i$ is the $\zeta^i$-eigenspace of $\sigma$. The eigenspace $L_0$ is a
simple Lie algebra, and there is a Cartan $\stackrel{\circ}{\mfh} \subset L$
compatible with $\sigma$ such that $\stackrel{\circ}{\mfh}_0 =
\stackrel{\circ}{\mfh} \cap L_0$ is a Cartan in $L_0$. The algebra $\mfh =
\stackrel{\circ}{\mfh}_0 \oplus \C c \oplus \C d$ is a Cartan for $\mfg$.  The
eigenspaces $L_i$ are irreducible $L_0$-modules. Choose a set of simple roots
$\alpha_1,\ldots,\alpha_l$ for $L_0$, and let $\psi$ be either the highest
weight of $L_1$ (if $k > 1$), or the highest root of $L_0$ (if $k=0$). Then
$\alpha_0 = d^* - \psi, \alpha_1,\ldots,\alpha_l$ is a set of simple roots for
$\mfg$, and $\alpha_0^{\vee} = c - \nu^{-1}(\psi),
\alpha_1^{\vee},\ldots,\alpha_l^{\vee}$ is a set of simple coroots, where $\nu
: \stackrel{\circ}{\mfh}_0 \arr \stackrel{\circ}{\mfh}_0^*$ is the isomorphism
defined by $\langle,\rangle$.  There is a unique real form $\mfh_{\R} =
\vspan_{\R}\{\alpha_i^{\vee}\} \oplus \R d$, and the anti-linear Cartan
involution sends $x z^{m} + \alpha c + \beta d \mapsto \overline{x} z^{-m}
-\overline{\alpha} c -\overline{\beta} d$, where $x \mapsto \overline{x}$ is
the anti-linear Cartan involution of $x$ in $L$. The real-valued symmetric
invariant bilinear form $\langle,\rangle$ on $\mfg$ is defined by
\begin{align*}
    & \langle x z^m, y z^n \rangle = \delta_{m,-n} \langle x,y\rangle,\   
    \langle c,d \rangle = a_0, \text{ and } \\ 
    & \langle x z^m, c \rangle = \langle x z^m, d \rangle =
    \langle c,c \rangle = \langle d,d \rangle = 0,
\end{align*}
where $a_0 = \langle \psi,\psi \rangle / 2$ (in fact, $a_0 = 1$ except when $L
= \mfsl(2l+1)$ and $k=2$, in which case $a_0 = 2$). The contragradient metric
$\{,\} = -\langle \cdot, \overline{\cdot} \rangle$ is positive-definite on
$\mfn$ as required.

The following lemma finishes the proof of Theorem \ref{T:brylkm}.
\begin{lemma}\label{L:cohvanish}
    Let $\mu$ be a dominant weight of an integrable highest weight
    $\mfg$-module $\mcL(\lambda)$, where $\lambda$ is a real-valued dominant
    weight and $\mfg$ is affine. If $\mu$ is dominant then
    $H^d_{cts}(\hat{\mfb},\mfh; \mcL(\lambda) \otimes S^* \hat{\mfn}^* \otimes
    \C_{-\mu}) = 0$ for all $d>0$.
\end{lemma}
\begin{proof}
    The result is trivial if $\lambda = \mu = 0$, so assume that $\lambda$ and
    $\mu$ have positive level. 

    $S^* \hat{\mfn}^*$ has a contragradient positive-definite Hermitian form
    from $\{,\}$. Since $\mu$ is a real-valued weight, $\C_{-\mu}$ has a
    contragradient positive-definite Hermitian form. Finally, $\mcL(\lambda)$
    has a contragradient positive-definite Hermitian form because $\lambda$ is
    a real-valued dominant weight.  Putting everything together, $V =
    \mcL(\lambda) \otimes S^* \hat{\mfn}^* \otimes \C_{-\mu}$ has a
    contragradient positive-definite Hermitian form. 

    The cohomology $H^*_{cts}(\hat{\mfb},\mfh; V)$ can be identified with the kernel
    of the Laplacian $\boxbar$ on the zero column $C^{0,*}(V)$ of the
    semi-infinite chain complex. By Nakano's identity, $\boxbar = \square +
    \degree + \curv$. $\square$ is positive semi-definite by definition.
    The curvature term splits into a sum $\curv = \curv_{\mcL(\lambda)} +
    \curv_{S^*} + \curv_{\C_{-\mu}}$. Since $\mcL(\lambda)$ is representation
    of $\mfg$, $\curv_{\mcL(\lambda)}$ is zero. Next consider $\curv_{S^*}
    + \deg$. We use the realisation of $\mfg$ via the loop algebra.  The
    contragradient metric $\{,\}$ induces a positive-definite metric on the
    loop algebra $\mfg' / \C c$, so we can pick a homogeneous basis for $\mfb$
    consisting of an orthonormal basis $\{y_s\}$ for the projection of $\mfb$
    to $\mfg' / \C c$, as well as $c$ and $d$.  The dual basis to
    $\{c,d,y_0,\ldots,y_s,\ldots\}$ is $\{a_0^{-1} d, a_0^{-1}
    c,-\overline{y_0}, \ldots,-\overline{y_s},\ldots\}$. Since $c$ is in the
    centre, we have $\ad^t(c) = 0$, so the terms $d_L(\ad^t(a_0^{-1} c))$ and
    $d_R(\ad^t(c) J)$ in $\curv_{S^*}$ are zero. Consequently
    \begin{equation*}
        \curv_{S^*}+\deg = \sum_{s \geq 0} d_L(\ad^t(-\overline{y_s})) 
                                        d_R(\ad^t(y_s) J) 
            = \sum_{s \geq 0} d_R(\ad^t(y_s) J)^* d_R(\ad^t(y_s) J)
    \end{equation*}
    is semi-positive. Finally we get that
    \begin{equation*}
        \curv_{\C_{-\mu}} = -\sum_{\alpha \in \Delta^+} \sum_{i, j}
            \eps(z^{i}_{\alpha}) \iota(z_{\alpha,j}) \mu([z_{\alpha,i},
                                                        \overline{z_{\alpha,j}}]),
    \end{equation*}
    where $z_{\alpha,i}$ runs through a basis for $\mfg_{\alpha}$ orthonormal
    in the Kahler metric. Now
    \begin{equation*}
        - \mu([z_{\alpha,i},\overline{z_{\alpha,j}}]) = \{z_{\alpha,i},z_{\alpha,j}\}
            \langle \mu,\alpha \rangle.
    \end{equation*}
    The result is that $\curv_{\C_{-\mu}}$ is a derivation which multiplies
    occurrences of $z_{\alpha}^j$ by the non-negative number
    $2 \langle \rho, \alpha \rangle \langle \mu,\alpha \rangle$, and thus is
    semi-positive.

    Now we look more closely at the kernel of $\boxbar$. The operator
    $\curv_{\C_{-\mu}}$ is strictly positive on $z^{\beta_1,i_1} \wedge \cdots
    \wedge z^{\beta_k,i_k} \otimes v$ unless all $\beta_i \in \Z[Y]$, where $Y
    = \{\alpha_i : \mu(\alpha_i^{\vee}) = 0 \}$. Let $A_Y$ be the submatrix of
    the defining matrix $A$ of $\mfg$ with rows and columns indexed by $\{ i :
    \alpha_i \in Y\}$. Recall that the Kac-Moody algebra $\mfg(A_Y)$
    defined by $A_Y$ embeds in $\mfg$. The standard nilpotent of $\mfg(A_Y)$
    is $\mfn_Y = \bigoplus_{\alpha \in \Delta^+ \cap \Z[Y]} \mfg_\alpha \subset
    \mfg$. Let $\mfu_Y = \bigoplus_{\alpha \in \Delta^+ \setminus \Z[Y]}
    \mfg_{\alpha}$. Since $\mu$ has positive level, $Y$ is a strict subset of
    simple roots, and since $\mfg$ is affine, $\mfg(A_Y)$ is finite-dimensional.
    Harmonic cocycles must belong to the kernel of $\curv_{\C_{-\mu}}$, so any
    harmonic cocycle $\omega$ must be in the $\mfh$-invariant part of
    \begin{equation*}
        \bigwedge^* \hat{\mfn}_Y^* \otimes S^* \hat{\mfn}^* \otimes \mcL(\lambda) 
            \otimes \C_{-\mu}.
    \end{equation*}
    As a vector space, this set can be identified with $\Omega_{pol}^*
    \hat{\mfn}_Y \otimes \C[\hat{\mfu}_Y] \otimes \mcL(\lambda)$, where
    $\Omega_{pol}^*$ is the ring of polynomial differential forms and
    $\hat{\mfu}_Y$ is pro-Lie algebra associated to $\mfu_Y$. For $\omega$ to
    be in the kernel of $\deg + \curv_{S^*}$, $\omega$ must lie in the kernel
    of the operators $d_R(\ad^t(y_s) J)$, $s \geq 0$. Since $d_R(\ad^t(c) J) =
    0$, we get that $d_R(\ad^t(x) J) \omega = 0$ for every $x \in \mfb_Y
    \subset \mfg' \cap \mfb$, where $\mfb_Y$ is the standard Borel of
    $\mfg(A_Y)$. Let $J^{-1}_{\Delta}$ denote the diagonal extension of
    $J^{-1}$ to $\bigwedge^* \hat{\mfn}^*$. Then $J^{-1}_{\Delta} \omega$
    vanishes under contraction by the vector fields $\mfn_Y \arr T \mfn_Y : x
    \mapsto (x,[x,y])$, $y \in \mfb$. At a point $x \in \mfn_Y$, these vector
    fields span the tangents to $\mcB_Y$-orbits. $\mfn_Y$ is the positive
    nilpotent of a finite-dimensional Kac-Moody, so $\mfn_Y$ has a dense
    $\mcB_Y$-orbit and thus $\omega$ must be of degree zero.
\end{proof}
The same proof applies with slight modification if $\mfg$ is a direct sum of
indecomposables of finite or affine type.

\section{A Brylinski filtration for indefinite Kac-Moody algebras}\label{S:indefinite}

In this section $\mfg$ will be an arbitrary symmetrizable Kac-Moody algebra.
Recall from the proof of Lemma \ref{L:cohvanish} that if $A$ is the defining
matrix of $\mfg$ and $Z$ is a subset of the simple roots then $A_Z$ refers to
the submatrix of $A$ with rows and columns indexed by $\{i : \alpha_i \in Z\}$.
\begin{prop}\label{P:affsupp}
    Let $\mfg$ be the symmetrizable Kac-Moody algebra defined by the generalized
    Cartan matrix $A$, and suppose $\mu$ is a dominant weight of an integrable
    highest weight representation $\mcL(\lambda)$, where $\lambda$ is real-valued.
    Write $\lambda - \mu = \sum k_i \alpha_i$, $k_i \geq 0$, and let $Z =
    \{ \alpha_i : k_i > 0\}$. If $A_Z$ is a direct sum of indecomposables of
    finite and affine type then $H^d_{cts}(\hat{\mfb},\mfh; S^* \hat{\mfn}^*
    \otimes \mcL(\lambda) \otimes \C_{-\mu}) = 0$ for $d>0$.
\end{prop}
Recall that the weight space $\mcL(\lambda)_{\mu}$ of an integrable highest
weight representation is filtered via polynomial degree on the isomorphic space
$(\mcL(\lambda) \otimes \C_{-\mu} \otimes \C[\mcU])^{\mcB}$. Let ${}^{deg}
P^{\lambda}_{\mu}(q)$ be the corresponding Poincare polynomial. Excepting
Proposition \ref{P:heisenberg}, the results of Sections \ref{S:bryldef} and
\ref{S:reduction} imply the following corollary:
\begin{cor}\label{C:affsupp} If the hypotheses of Proposition \ref{P:affsupp}
    hold then $m^{\lambda}_{\mu}(q) = {}^{deg} P^{\lambda}_{\mu}(q)$
\end{cor}
The conclusions of Theorem \ref{T:brylkm} hold similarly, with the Brylinski
filtration replaced by the degree filtration. 

The requirement in Proposition \ref{P:affsupp} and Corollary \ref{C:affsupp}
that $\lambda - \mu$ have affine support is a technical assumption used to
prove the positive-definiteness of the $\deg + \curv_{S^*}$ term in the
Laplacian. It is unclear to the author whether or not this hypothesis
can be dropped. 

\begin{proof}[Proof of Proposition \ref{P:affsupp}]
    We continue to use the notation of Section \ref{S:cohvanish}.  For
    instance, $V = S^* \hat{\mfn}^* \otimes \mcL(\lambda) \otimes \C_{-\mu}$.
    Recall that $\boxbar = \square + \deg + \curv_V$, and $\curv_V =
    \curv_{\mcL(\lambda)} + \curv_{\C_{-\mu}} + \curv_{S^*}$. The operators
    $\square$, $\curv_{\mcL(\lambda)}$, and $\curv_{\C_{-\mu}}$ are positive
    semi-definite as before, while
    \begin{equation*}
        \deg + \curv_{S^*} = \sum_{k \geq 1} d_R(\ad^t(x_k) J)^* d_R(\ad^t(x_k) J) 
            + \sum_{i} d_L(\ad^t(u^i)) d_R(\ad^t(u_i) J),
    \end{equation*}
    where $\{x_k\}$ is a basis for $\mfn$ orthonormal in the contragradient
    metric, and $\{u_i\}$ and $\{u^i\}$ are dual bases for $\mfh$. The first
    summand in this equation is positive semi-definite, but the second is not
    if there are roots with $\langle \alpha, \alpha \rangle < 0$. Indeed,  
    writing
    \begin{multline}\label{E:switch}
        \sum_{i}  d_L(\ad^t(u^i)) d_R(\ad^t(u_i) J) = \\
             \sum_{i} \Switch(\ad^t(u^i) J, \ad^t(u_i)) + \sum_i (\ad^t(u^i) \ad^t(u_i) J)^{\wedge},
    \end{multline}
    we see that the first summand in Equation (\ref{E:switch}) is the
    second order operator defined by 
    \begin{equation*}
         x \otimes y \mapsto \frac{ \langle \alpha, \beta \rangle }{2 \langle \rho,\alpha \rangle}
            y \otimes x, x \in \mfg_{\alpha}^*, y \in \mfg_{\beta}^*, 
    \end{equation*}
    while the second summand in Equation (\ref{E:switch}) is the derivation of
    $\bigwedge^* \hat{\mfn}^*$ induced by the map
    \begin{equation*}
        x \mapsto \frac{ \langle \alpha, \alpha \rangle}{2 \langle \rho, \alpha \rangle} x, 
            x \in \mfg_{\alpha}^*
    \end{equation*}
    on $\hat{\mfn}^*$. 
        
    Let $\mfg(A_Z)$ be the corresponding Kac-Moody subalgebra of $\mfg$, and
    let $\mfn_Z$ be the standard nilpotent. $\mfg(A_Z)$ has a Cartan subalgebra
    $\mfh_Z \subset \mfh$, and the real-valued non-degenerate symmetric
    invariant form on $\mfg$ restricts to such a form on $\mfg(A_Z)$. Any
    $\mfh$-invariant element of $\bigwedge^* \hat{\mfn}^* \otimes V$ must
    belong to $\bigwedge^* \hat{\mfn}_Z^* \otimes S^* \hat{\mfn}_Z^* \otimes
    \mcL(\lambda) \otimes \C_{-\mu}$. We claim that the operator $\sum_{i}
    d_L(\ad^t(u^i)) d_R(\ad^t(u_i) J)$ on $\bigwedge^* \hat{\mfn}^* \otimes S^*
    \hat{\mfn}^*$ restricts on $\bigwedge^* \hat{\mfn}_Z^* \otimes S^*
    \hat{\mfn}_Z^*$ to the operator $\sum_{i} d_L(\ad^t(v^i)) d_R(\ad^t(v_i)
    J)$, where $\{v_i\}$ and $\{v^i\}$ are dual bases of $\mfh_Z$. To verify
    this claim, note that a choice of symmetric invariant form corresponds to a
    choice of a diagonal matrix $D$ with positive diagonal entries, such that
    $D A$ is a symmetric matrix. If $x \in \mfh^*$ the invariant form satisfies
    $\langle x, \alpha_i \rangle = D_{ii} x(\alpha_i^{\vee})$. The operator in
    Equation (\ref{E:switch}) thus depends only on $A$ and $D$; the claim
    follows from the observation that the action of the operator on
    $\bigwedge^* \hat{\mfn}_Z^* \otimes S^* \hat{\mfn}_Z^*$ depends only on
    $A_Z$ and $D_Z$. 

    Now suppose $A_Z$ is a direct sum of indecomposables of finite and affine
    type. The operator $\sum_{i} d_L(\ad^t(v^i)) d_R(\ad^t(v_i) J)$ decomposes
    into a summand for each component, each of which is positive semi-definite
    as in the proof of Lemma \ref{L:cohvanish}. We finish as in the proof of
    Lemma \ref{L:cohvanish}, but taking $Y = \{ \alpha_i \in Z : \mu(\alpha_i)
    = 0\}$. 
\end{proof}

\bibliographystyle{plain}

\end{document}